\documentclass[12pt]{amsart}
\usepackage{amsmath,amsthm,amssymb,bbm,dsfont}
\usepackage{graphicx}
\usepackage{subfig}
\usepackage{cite,float}

\numberwithin{equation}{section}
\newtheorem{theorem}{Theorem}[section]
\newtheorem{lemma}[theorem]{Lemma}
\newtheorem{result}[theorem]{Result}

\newtheorem{proposition}[theorem]{Proposition}
\theoremstyle{remark}

\makeatletter
\@namedef{subjclassname@2010}{%
	\textup{2020} Mathematics Subject Classification}
\makeatother
\begin{document}
		\title{On the Kobayashi isometries of a class of 2-dimensional Lempert 
			manifolds}
	
	\author[N. Gupta]{Naveen Gupta}
	\address{Department of Mathematics, Netaji Subhas University of Technology,
		Delhi--110 078, India}
	\email{naveen.gupta@nsut.ac.in, ssguptanaveen@gmail.com}

	\begin{abstract}
		We prove that every Kobayashi distance isometry between $2$-dimensional 
		Lempert manifolds having finite universal set (with atleast three elements) is 
		holomorphic or anti-holomorphic. It is shown that every such isometry is 
		$C_{loc}^{1,1}$ without any regularity assumption on the isometry. 
	\end{abstract}
	\maketitle

	\section{Introduction}
	
	Let $M,\, N$ be two Lempert manifolds with finite Carath\`{e}odory universal sets. Chavan 
	and Zwonek \cite[~Theorem 2]{cha-zwo} proved the following result:
	\begin{result}\label{res:cha-zwo}
		Let $M,N$ be $2$-dimensional Lempert manifolds such that
		their Carath\`{e}odory universal set is finite but has at least three elements.
		Let $F : M \to N$ be a $C^1$-smooth Kobayashi metric isometry. Then $F$
		is (anti)holomorphic.
	\end{result}

For a $C^1$ Kobayashi distance  $F$ it is follows that $F$ is $C^1$ 
Kobayashi metric isometry and hence Chavan and Zwonek's result proves  that $F$ is holomorphic 
or anti-holomorphic. We prove 
the following theorem to establish that the $C^1$ assumption is not required here.

\begin{theorem}\label{thm:main}
	Every Kobayashi distance isometry between $2$-dimensional Lempert manifolds with 
	finite  Carath\`{e}odory universal set with atleast three elements 
	 is either holomorphic or anti-holomorphic.
\end{theorem}

The proof follows the idea presented in \cite{Edigarian}, in which the author proves that 
the $C^1$ assumption in \cite[~Theorem 1]{cha-zwo1} is superflous.
 
% We prove the following proposition to prove the above stated theorem.

	\section{preliminaries and notations}

We begin by introducing some basic concepts and a few results associated to them. 
These will help us in discussing the proof of the main 
results.

Let $M$ be a complex manifold. A Finsler metric on $M$ is a continuous function 
$F:TM\to [0,\infty)$ whose restriction to each tangent space $T_pM$ is a Minkowski norm.
In Finsler geometry, $F$ is required to be $C^2$. But Matveev-Trayonov studied Finsler 
metrics of low regularity. A Finsler metric $F$ is $C^{k,\alpha}_{loc}$ ($k\in \mathbb{N},
\, 0\leq \alpha\leq 1$) if in local co-ordinates it is of class $C^{k,\alpha}_{loc}.$

Carath\'{e}odory and Kobayashi distance on $M$ are denoted by $c_M,\, k_M$, whereas 
their infinitesimal form Carath\'{e}odory-Rieffen and Kobayashi-Royden metric are 
denoted by $\gamma_M, \varkappa_M$ respectively.

A connected complex manifold $M$ is said to be Lempert if it is Taut and the equality
$$c_M\equiv l_M\equiv k_M,\, \gamma_M\equiv \varkappa_M$$ holds, where $l_M$ denotes
the Lempert function on $M.$ It is an immediate consequence of the definition that 
a Lempert manifold is complete with the Kobayashi distance (or equivalently 
Carath\'{e}odory distance).

% We denote by $\mathbb{D}$ the unit disc in the 
%complex plane $\mathbb{C}$ and the P\'{o}incare distance on $D$ is denoted by $\rho.$ 
For two Lempert manifolds $M,N$, a map $F:M\to N$  is said to be Kobayashi distance isometry if 
$$k_M(z,w)=k_{N}(F(z),F(w))$$ for all $z,w\in M.$

A $C^1$ map $F:M\to N$ is said to be Kobayashi metric isometry if 
$$\varkappa_M(z,w)=\varkappa_N(F(z),F(w))$$ for all $z,w\in M$. 

Let us quickly recall the definition of Carath\'{e}odory distance on $M$. For any 
$z,w\in M$, $$c_M(z,w):=\sup_{f\in \mathcal{F}}\{\rho(F(z),F(w))\},$$ 
where $\mathcal{F}$ consists of holomorphic maps from $M$ to the unit disc $\mathbb{D}$
in the complex plane and $\rho$ denotes the P\'{o}incare distance on $\mathbb{D}$. 
The smallest subset $\mathcal{F}_M$ of holomorphic maps $f:M\to \mathbb{D}$ such that 
above stated expression for $c_M$ holds is said to be Carath\'{e}odory universal set. 
In case, $F_M$ consists of finite number of functions, it is said to be finite 
Carath\'{e}odory universal set.

\section{Proof of the main result}

We begin with the proof of the following key lemma.

\begin{lemma}\label{lem:connect}
Let $M$ be a Lempert manifold whose Carath\`{e}odory universal set 
$\mathcal{F}=\{f_1,f_2,\ldots,f_m\}$ is finite. Then the  Carath\`{e}odory-Rieffen 
metric $\gamma_M:TM\to [0,\infty)$ is $C_{\text{loc}}^{0,1}$.
\end{lemma}

\begin{proof}
	Recall that  for any $p\in M$ and $X\in T_pM$, we have 
	$$\gamma_M(p;X)=\sup_{f\in \mathcal{F}}\gamma_{\mathbb{D}}(f(p);f'(p)X).$$
	
	Since the Carath\`{e}odory universal set is finite, the expression reduces to 
		$$\gamma_M(p;X)=\max_{1\leq i\leq m}\gamma_{\mathbb{D}}(f_i(p);f_i'(p)X).$$
		
		For any fixed $1\leq i\leq m$,  consider the map 
		\begin{align*}(p,X)\to& \gamma_{\mathbb{D}}(f_i(p);f_i'(p)X)\\
			& =\dfrac{|f_i'(p)X|}{1-|f_i(p)|}.			
		\end{align*}
	Since $f_i$ is holomorphic and the image of $f_i$ lies within $\mathbb{D}$, it follows that 
	this map is locally Lipschitz. Therefore, $\gamma_M$, which is maximum of these Lipschitz 
	functions, is also Lipschitz. This proves that $\gamma_M$ is $C_{\text{loc}}^{0,1}$.
\end{proof}

We now prove the following proposition which in turn will imply Theorem \ref{thm:main}.

 \begin{proposition}\label{prop:connect}
	Every Kobayashi distance isometry between Lempert manifolds with finite 
	Carath\`{e}odory universal set is $C_{\text{loc}}^{1,1}$.
\end{proposition}

\begin{proof}
	Let $F:M\to N$ be Kobayashi distance isometry. Observe that for any $z\in M$ and 
	$X\in T_zM$, the map
	$$X\to \gamma_M(z;X)$$ is always a seminorm. Since $M$ is Taut, it becomes a norm.
	 Now, using Lemma \ref{lem:connect}, 
	$(M,\gamma_M)$ and $(N,\gamma_N)$ are $C_{\text{loc}}^{0,1}$ Finsler manifolds in 
	the sense of Matveev-Trayanov. Since $F$ is an isometry, it is continuous and injective. 
	
	Since it is injective, $F(M)\subseteq N$ is open. Also, $(M,\gamma_{M})$ is complete. Therefore, 
	$F(M)\subset N$ is complete and thus it is closed. This implies that $F(M)=N$ and 
	that it is bijective.  Now, using the result \cite[~Corollary C]{matveev} (for $k=0, \alpha =1$), we obtain that 
	$F$ is $C_{\text{loc}}^{1,1}$.  
\end{proof}

Theorem \ref{thm:main} now follows by combining Proposition \ref{prop:connect} with 
\cite[~Theorem 2]{cha-zwo}.

\begin{proof}[Proof of Theorem \ref{thm:main}]
	Let $M,\,N$ be $2$-dimensional Lempert manifolds with finite Carath\'{e}odory
	 universal sets with atleast three elements. Let $F:M\to N$ be Kobayashi distance
	  isometry. Proposition 
	 \ref{prop:connect} gives us that $F$ is $C_{\text{loc}}^{1,1}$ and hence $C^1$ 
	 in particular. Result \ref{res:cha-zwo} now completes the proof. 
\end{proof}

\end{document}